\documentclass[graybox]{svmult}

\usepackage{mathptmx}       
\usepackage{helvet}         
\usepackage{courier}        
\usepackage{type1cm}        
\usepackage{makeidx}         
\usepackage{graphicx}        
\usepackage{multicol}        
\usepackage[bottom]{footmisc}

\makeindex             

\usepackage{amsmath,amssymb,enumerate}
\smartqed
\DeclareMathOperator{\indicatorfMRS}{\mathsf{1}}
\DeclareMathOperator{\MEMRS}{\mathsf{E}}
\DeclareMathOperator{\varMRS}{var}

\DeclareMathOperator{\covMRS}{cov}
\DeclareMathOperator{\pdfMRS}{pdf}
\DeclareMathOperator{\constoneMRS}{\indicatorMRS_{[0,T]}}
\newcommand{\BetaMRS}{\mathrm B}
\newcommand{\indicatorMRS}{\indicatorfMRS\nolimits}
\newcommand{\NormalMRS}{\mathcal N}
\newcommand*{\normMRS}[1]{\left\lVert#1\right\rVert}
\newcommand*{\absMRS}[1]{\left\lvert#1\right\rvert}
\newcommand*{\setMRS}[1]{\left\{#1\right\}}

\begin{document}

\title*{Parameter estimation for Gaussian processes with application to the model with two independent fractional Brownian motions}
\titlerunning{Parameter estimation for Gaussian processes}
\author{Yuliya Mishura, Kostiantyn Ralchenko and Sergiy Shklyar}
\institute{Yuliya Mishura \at
Department of Probability Theory, Statistics and Actuarial Mathematics,
  Taras Shevchenko National University of Kyiv,
  64 Volodymyrska,
  01601 Kyiv, Ukraine
\email{myus@univ.kiev.ua}
\and
Kostiantyn Ralchenko \at
Department of Probability Theory, Statistics and Actuarial Mathematics,
  Taras Shevchenko National University of Kyiv,
  64 Volodymyrska,
  01601 Kyiv, Ukraine
\email{k.ralchenko@gmail.com}
\and
Sergiy Shklyar \at
Department of Probability Theory, Statistics and Actuarial Mathematics,
  Taras Shevchenko National University of Kyiv,
  64 Volodymyrska,
  01601 Kyiv, Ukraine
\email{shklyar@univ.kiev.ua}
}
%
%
\maketitle

\abstract*{The purpose of the article is twofold.
Firstly, we review some recent results on the maximum likelihood estimation in the regression model of the form $X_t = \theta G(t) + B_t$, where $B$ is a  Gaussian process, $G(t)$ is a known function, and $\theta$ is an unknown drift parameter. The estimation techniques for the cases of discrete-time and continuous-time observations are presented. As examples, models with fractional Brownian motion, mixed fractional Brownian motion, and sub-fractional Brownian motion are considered.
Secondly, we study in detail the model with two independent fractional Brownian motions and apply the general results mentioned above to this model.}

\abstract{The purpose of the article is twofold.
Firstly, we review some recent results on the maximum likelihood estimation in the regression model of the form $X_t = \theta G(t) + B_t$, where $B$ is a  Gaussian process, $G(t)$ is a known function, and $\theta$ is an unknown drift parameter. The estimation techniques for the cases of discrete-time and continuous-time observations are presented. As examples, models with fractional Brownian motion, mixed fractional Brownian motion, and sub-fractional Brownian motion are considered.
Secondly, we study in detail the model with two independent fractional Brownian motions and apply the general results mentioned above to this model.
}

\section{Introduction}\label{MRS:sec:1}\label{MRS:sec:problem}

Gaussian processes with drift arise in many applied areas, in particular, in telecommunication and on financial markets.
An observed process often can be decomposed as the sum of a useful signal and a random noise, where the last one mentioned is usually modeled by a centered Gaussian process, see, e.\,g., \cite[Ch.~VII]{MRS:IbrRoz}.

The simplest example of such model is the process
\[
Y_t=\theta t+W_t,
\]
where $W$ is a Wiener process.
In this case the MLE of the drift parameter $\theta$ by observations of $Y$ at points $0=t_0\le t_1\le\ldots\le t_{N}=T$ is given by
\[
\hat \theta=\frac{1}{t_{N}-t_0}\sum_{i=0}^{N-1}
\left(Y_{t_{i+1}}-Y_{t_i}\right)
=\frac{Y_{T}-Y_{0}}{T},
\]
and depends on the observations at two points, see e.\,g.\ \cite{MRS:BTT}.
Models of such type are widely used in finance.
For example, Samuelson's model \cite{MRS:samuelson1965rational} with constant drift parameter $\mu$ and known volatility $\sigma$ has the form
\[
\log S_t=\left(\mu-\frac{\sigma^2}{2}\right)t+\sigma W_t,
\]
and the MLE of $\mu$ equals
\[
\hat\mu=\frac{\log S_T-\log S_0}{T}+\frac{\sigma^2}{2}.
\]

At the same time, the model with Wiener process is not suitable for many processes in natural sciences, computer networks, financial markets, etc., that have long- or short-term dependencies, i.\,e., the correlations of random noise in these processes decrease slowly with time (long-term dependence) or rapidly with time (short-term dependence).
In particular, the models of financial markets demonstrate  various kinds of memory (short or long).
However, a Wiener process has independent increments, and, therefore, the random noise generated by it is ``white'', i.\,e., uncorrelated.  The most simple way to overcome this limitation is to use fractional Brownian motion.
In some cases even more complicated models are needed.
For example, the noise can be modeled by mixed fractional Brownian motion \cite{MRS:cheredito}, or by the sum of two fractional Brownian motions \cite{MRS:MiSumfbms1}.
Moreover, recently Gaussian processes with non-stationary increments have become popular such as sub-fractional \cite{MRS:BGT}, bifractional \cite{MRS:HV} and multifractional \cite{MRS:benassi,MRS:peltier1995,MRS:ralshev} Brownian motions.

In this paper we study rather general model where the noise is represented by a centered Gaussian process $B=\{B_t, t\ge 0\}$ with known covariance function, $B_0 = 0$. We assume that all finite-dimensional distributions of the process $\{B_t, \; t>0\}$ are multivariate normal distributions
with nonsingular covariance matrices.
We observe the process $X_t$ with a drift $\theta G(t)$,
that is,
\begin{equation}\label{MRS:eq:model}
X_t = \theta G(t) + B_t,
\end{equation}
where   $G(t) = \int_0^t g(s)\, ds, $ and $g\in L_1[0,t]$ for any $t>0$.


The paper is devoted to the estimation of the parameter $\theta$ by observations of the process $X$.
We consider the MLEs for discrete and continuous schemes of observations.
The results presented are based on the recent papers \cite{MRS:NA,MRS:AJS}.
Note that in \cite{MRS:AJS} the model~\eqref{MRS:eq:model} with $G(t)=t$ was considered, and the driving process $B$ was a process with stationary increments. Then in \cite{MRS:NA} these results were extended to the case of non-linear drift and more general class of driving processes.
In the present paper we apply the theoretical results mentioned above to the models with fractional Brownian motion, mixed fractional Brownian motion and sub-fractional Brownian motion.

Similar problems for the model with linear drift driven by fractional Brownian motion were studied in \cite{MRS:BTT,MRS:HNXZ11,MRS:Breton98,MRS:nvv}.
The mixed Brownian\,---\,fractional Brownian model was treated in~\cite{MRS:CaiChigKlept}.
In \cite{MRS:BerWol,MRS:Privault} the nonparametric functional estimation of the drift of a Gaussian processes was considered (such estimators for fractional and subfractional Brownian motions were studied in \cite{MRS:EOO} and \cite{MRS:ShenYan} respectively).

In the present paper special attention is given to the model of the form\linebreak
$X_t=\theta t + B^{H_1}_t+ B^{H_2}_t$
with two independent fractional Brownian motions $B^{H_1}$ and $B^{H_2}$.
This model was first studied in \cite{MRS:MiSumfbms1}, where a strongly consistent estimator for the unknown drift parameter $\theta$ was constructed for $1/2<H_1<H_2<1$ and $H_2-H_1>1/4$ by continuous-time observations of $X$. Later, in \cite{MRS:MishuraVoronov}, the strong consistency of this estimator was proved for arbitrary $1/2<H_1<H_2<1$.
The details on this approach are given in Remark~\ref{MRS:rem:mish-vor} below.
However, the problem of drift parameter estimation by discrete observations in this model was still open.
Applying our technique, we obtain the discrete-time estimator of
$\theta$ and prove its strong consistency for any $H_1, H_2\in(0,1)$. Moreover, we also construct the continuous-time estimator and prove the convergence of the discrete-time estimator to the continuous-time one in the case where $H_1\in(1/2,3/4]$ and $H_2\in(H_1,1)$.

It is worth mentioning that the drift parameter estimation is developed for more general models involving fBm.
In particular, the fractional Ornstein--Uhlenbeck process is a popular and well-studied model with fBm. The MLE of the drift parameter for this process
was constructed in \cite{MRS:KleptsynaLeBreton} and further investigated in \cite{MRS:Bercu2010,MRS:Tanaka13,MRS:TudorViens}.
Several non-standard estimators for the drift parameter of an  ergodic fractional Ornstein--Uhlenbeck process were proposed in \cite{MRS:HuNu} and studied in \cite{MRS:HuNuZhou}. The corresponding non-ergodic case was treated in \cite{MRS:BESO,MRS:EMESO,MRS:Tanaka15}.
In the papers \cite{MRS:CES,MRS:Es-Sebaiy,MRS:ESN,MRS:HuSong13,MRS:xiaoZX,MRS:xiaoZZ} drift parameter estimators were constructed via discrete observations.
More general fractional diffusion models were studied in \cite{MRS:KMM,MRS:mishura} for continuous-time estimators and in \cite{MRS:kumirase,MRS:mira,MRS:mirasesh} for the case of discrete observations.
An estimator of the volatility parameter was constructed in \cite{MRS:kumiva}. For Hurst index estimators see, e.\,g., \cite{MRS:kubmish} and references cited therein.
Mixed diffusion model including fractional Brownian motion and Wiener process was investigated in \cite{MRS:KMM}.
We refer to the paper \cite{MRS:Mishura2017} for a survey of the results on parameter estimation in fractional and mixed diffusion models and to the books \cite{MRS:KMR-book,MRS:prakasa_rao2010} for a comprehensive study of this topic.

The paper is organized as follows.
In Section~\ref{MRS:sec:discr} we construct the MLE by discrete-time observations and formulate the conditions for its strong consistency.
In Section~\ref{MRS:sec:cont} we consider the estimator constructed by continuous-time observations and the relations between discrete-time and continuous-time estimators.
In Section~\ref{MRS:sec:examples} these results are applied to various models mentioned above.
In particular, the new approach to parameter estimation in the model with two independent fractional Brownian motions is presented in Subsection \ref{MRS:sscec:2fBm}.
Auxiliary results are proved in the appendices.

\section{Construction of drift parameter estimator for discrete-time observations}
\label{MRS:sec:discr}

Let the process $X$ be observed at the points  $0<t_1<t_2<\ldots<t_N$.
Then the vector of increments
\[
\Delta X^{(N)} = (X_{t_1}, \: X_{t_2}-X_{t_1}, \:
\ldots, \: X_{t_N}-X_{t_{N-1}})^\top
\]
is a one-to-one function of the observations.
We assume in this section that the inequality $G(t_k) \neq 0$ holds at least for one $k$.

Evidently,  vector $\Delta X^{(N)}$ has Gaussian distribution
$\NormalMRS(\theta \Delta G^{(N)}, \Gamma^{(N)})$, where
\[
\Delta G^{(N)} = \bigl(G(t_1), \: G(t_2)-G(t_1), \:
\ldots, \: G(t_N) - G(t_{N-1})\bigr)^\top.
\]
Let $\Gamma^{(N)}$ be the covariance matrix of the vector
\[
\Delta B^{(N)} = (B_{t_1}, \: B_{t_2}-B_{t_1}, \:
\ldots, \: B_{t_N}-B_{t_{N-1}})^\top.
\]
The density of the distribution of $\Delta X^{(N)}$
w.\,r.\,t. the Lebesgue measure is
\[\textstyle
\pdfMRS_{\Delta X^{(N)}}(x)
=
\frac{(2 \pi)^{-N/2}}{\sqrt{\det \Gamma^{(N)}}}
\exp\left\{ - \frac{1}{2}
\left(x - \theta \Delta G^{(N)}\right)^\top
\left(\Gamma^{(N)}\right)^{-1}
\left(x - \theta \Delta G^{(N)}\right)
\right\}.
\]
Then one can take the density of the distribution of the vector  $\Delta X^{(N)}$ for a given $\theta$
w.\,r.\,t. the density for $\theta=0$ as a likelihood function:
\begin{equation}\label{MRS:eq:LN}
L^{(N)}(\theta) =
\exp\left\{
\theta (\Delta G^{(N)})^\top (\Gamma^{(N)})^{-1} \Delta X^{(N)} -
\frac{\theta^2}{2}
(\Delta G^{(N)})^\top (\Gamma^{(N)})^{-1} \Delta G^{(N)}
\right\}.
\end{equation}
The corresponding MLE equals
\begin{equation}\label{MRS:eq:discr-mle}
\hat\theta^{(N)} = \frac {\left(\Delta G^{(N)}\right)^\top
\left(\Gamma^{(N)}\right )^{-1} \Delta X^{(N)}} {\left(\Delta
G^{(N)}\right)^\top \left(\Gamma^{(N)}\right)^{-1} \Delta
G^{(N)}}.
\end{equation}

\begin{theorem}[Properties of the discrete-time MLE \cite{MRS:NA}]
\label{MRS:thm:consdiscrL2}
1.
The estimator $\hat\theta^{(N)}$ is unbiased and normally distributed:
\[
\hat\theta^{(N)} - \theta \simeq \NormalMRS\left(0,\frac{1}
{(\Delta G^{(N)})^\top (\Gamma^{(N)})^{-1} \Delta G^{(N)}}\right).
\]
2.
Assume that
\begin{equation}\label{MRS:eq:condconssi}
\frac{\varMRS B_{t}}{G^2(t)} \to 0, \quad\text{as } t\to\infty.
\end{equation}
If $t_{N} \to \infty$, as $N\to\infty$,
then the discrete-time MLE
 $\hat\theta^{(N)}$ converges to $\theta$ as $N\to\infty$ almost surely and in $L_2(\Omega)$.
\end{theorem}

\section{Construction of drift parameter estimator for continuous-time observations}
\label{MRS:sec:cont}

In this section we suppose that the process $X_t$ is observed on the whole interval $[0,T]$. We investigate MLE for the parameter $\theta$ based on these observations.

Let $\langle f,\, g\rangle = \int_0^T f(t) g(t) \, dt$.
Assume that the function $G$ and the process $B$ satisfy the following conditions.
\begin{enumerate}[\bf(A)]
\item There exists a linear self-adjoint operator
$\Gamma=\Gamma_T : L_2[0,T] \to L_2[0,T]$
such that
\begin{equation}\label{MRS:eq:ass10st}
\covMRS(X_s, X_t) = \MEMRS B_s B_t =
\int_0^t \Gamma_T \indicatorMRS_{[0,s]}(u) \, du =
\langle \Gamma_T \indicatorMRS_{[0,s]}, \, \indicatorMRS_{[0,t]} \rangle.
\end{equation}
\item The drift function $G$ is not identically zero, and in its representation $G(t) = \int_0^t g(s)\, ds$  the function  $g \in L_2[0,T]$.
\item
There exists a function $h_T \in L_2[0,T]$ such that
$g = \Gamma h_T$.
\end{enumerate}
Note that under assumption $(A)$ the covariance between integrals of deterministic functions
$f \in L_2[0,T]$ and $g \in L_2[0,T]$
w.\,r.\,t. the process $B $ equals
\[
\MEMRS \int_0^T f(s)\, dB_s \, \int_0^T g(t)\, dB_t =
\langle \Gamma_T f,\,   g\rangle .
\]

\begin{theorem}[Likelihood function and continuous-time MLE \cite{MRS:NA}]
\label{MRS:th-L}
Let $T$ be fixed,
  assumptions $(A)$--$(C)$ hold.
Then one can choose
\begin{equation}\label{MRS:eq-L}
L(\theta) = \exp\left\{
\theta \int_0^T h_T(s) \, dX_s - \frac{\theta^2}{2}
\int_0^T g(s) h_T(s) \, ds\right\}
\end{equation}
as a likelihood function.
The MLE equals
\begin{equation}\label{MRS:eq:defthetaC}
\hat\theta_T = \frac
{\int_0^T h_T(s) \, dX_s}
{\int_0^T g(s) h_T(s) \, ds} .
\end{equation}
It is unbiased and normally distributed:
\[
\hat\theta_T-\theta\simeq\NormalMRS\left(0,\frac{1}
{\int_0^T g(s) h_T(s) \, ds} \right).
\]
\end{theorem}

\begin{theorem}[Consistency of the continuous-time MLE \cite{MRS:NA}]
\label{MRS:thm:consTL2}
Assume that assumptions $(A)$--$(C)$ hold for all $T>0$.
If, additionally,
\begin{equation}\label{MRS:eq:ccondli0}
\liminf_{t\to\infty} \frac{\varMRS B_t}{G(t)^2} = 0,
\end{equation}
then the estimator $\hat\theta_T$
converges to $\theta$ as $T\to \infty$ almost surely and in mean square.
\end{theorem}

\begin{theorem}[Relations between discrete and continuous MLEs \cite{MRS:NA}]\label{MRS:thm:thaxL2}
Let the assumptions of Theorem~\ref{MRS:th-L} hold.
Construct the estimator $\hat\theta^{(N)}$ from \eqref{MRS:eq:discr-mle}
by observations $X_{Tk/N}$, $k=1,\ldots,N$.
Then
\begin{enumerate}
\item
the estimator $\hat\theta^{(N)}$ converges to $\hat\theta_T$
in mean square, as $N\to\infty$,
\item
the estimator $\hat\theta^{(2^n)}$ converges to $\hat\theta_T$
almost surely, as $n\to\infty$.
\end{enumerate}
\end{theorem}

\section{Application of estimators to models with various noises}
\label{MRS:sec:examples}

\subsection{The model with fractional Brownian motion and linear drift}
\begin{definition}
The \emph{fractional Brownian motion}\index{fractional Brownian motion} $B^H=\setMRS{B^H_t,t\ge0}$ with Hurst index $H\in(0,1)$ is a centered Gaussian process with $B_0=0$ and covariance function
\[
\MEMRS B^H_t B^H_s = \tfrac12\left(t^{2H}+s^{2H}-\absMRS{t-s}^{2H}\right).
\]
\end{definition}

Let $H\in(0,1)$ be fixed.
Consider the model
\begin{equation}\label{MRS:eq:model-fBm}
X_t=\theta t+B^H_t.
\end{equation}
where $X$ is an observed stochastic process,
$B^H$ is an unobserved fractional Brownian motion
with Hurst index $H$, and
$\theta$ is a parameter of interest.
Any finite slice of the stochastic process $\{B^H_t, \; t>0\}$
has a multivariate normal distribution with nonsingular
covariance matrix.
Since $\varMRS\left(B^H_t\right) =t^{2H}$,
the random process $B^H$ satisfies
Theorem~\ref{MRS:thm:consdiscrL2}.
Hence, we have the following result.
\begin{corollary}\label{MRS:cor:fBm}
Under condition $t_N\to +\infty$ as $N\to\infty$,
the estimator $\hat\theta^{(N)}$ in the model \eqref{MRS:eq:model-fBm} is $L_2$-consistent and
strongly consistent.
\end{corollary}

\begin{remark}
Bertin et al. \cite{MRS:BTT} considered the MLE in the model \eqref{MRS:eq:model-fBm} in the discrete scheme of observations, where the trajectory of $X$ was observed at the points
$t_k=\frac{k}{N}$, $k=1,2,\ldots,N^\alpha$, $\alpha>1$.
Hu et al.\ \cite{MRS:HNXZ11} investigated the MLE by discrete observations at the points $tk=kh$, $k=1,2,\ldots,N$. They considered even more general model of the form $X_t=\theta t+\sigma B^H_t$ with unknown $\sigma$.
In both papers $L_2$-consistency and
strongly consistency of the MLEs were proved.
Note that in Corollary~\ref{MRS:cor:fBm} both these schemes of observations are  allowed, since the only condition $t_N\to\infty$ is required.
\end{remark}

Now we consider the case of continuous-time observations and apply the results of Section \ref{MRS:sec:cont} to the model \eqref{MRS:eq:model-fBm}. Let $H\in(\frac12,1)$.
Denote by $\Gamma_H$ the corresponding operator $\Gamma$ for the model \eqref{MRS:eq:model-fBm}.
Then
\begin{equation}\label{MRS:eq:Gamma_H}
    (\Gamma_H f)(t) = H(2H-1)\int_0^T\frac{f(s)}{\absMRS{t-s}^{2-2H}}\,ds.
\end{equation}
For the function
\[
h_T(s) = C_H
s^{1/2-H} (T-s)^{1/2-H},
\]
$C_H=\left(H (2H-1) \mathrm{B} \left(H-\frac12, \frac32-H\right)\right)^{-1}$,
we have that
\begin{equation}\label{MRS:eq:GHhT1}
\Gamma_H h_T = \constoneMRS,
\end{equation}
see~\cite{MRS:nvv}.
The MLE is given by
\[
\hat\theta_T = \frac{T^{2H - 2}}{\BetaMRS(3/2-H,\: 3/2-H)}
\int_0^T s^{1/2-H} (T-s)^{1/2-H} \, dX_s.
\]
This estimator was studied in \cite{MRS:Breton98,MRS:nvv}, see also \cite[Example~3.11]{MRS:AJS}.

\begin{corollary}
Let $H\in(\frac12,1)$. The conditions of
Theorems \ref{MRS:th-L}, \ref{MRS:thm:consTL2} and \ref{MRS:thm:thaxL2},
are satisfied.
The estimator $\hat\theta_T$ is $L_2$-consistent
and strongly consistent.
For fixed $T$, it can be approximated by discrete-sample estimator
in mean-square sense.
\end{corollary}

\subsection{Model with fractional Brownian motion and power drift}
Now we generalize the model \eqref{MRS:eq:model-fBm} for the case of the non-linear drift function $G(t)=t^{\alpha + 1}$.
Let $0 < H < 1$ and
$\alpha >  -1$.
Consider the process
\begin{equation}\label{MRS:eq:ex-fbm}
X_t = \theta t^{\alpha + 1} + B_t^H
\end{equation}
This is a particular case of model~\eqref{MRS:eq:model},
with $g(t) = (\alpha + 1) t^\alpha$.

Now verify the conditions of the theorems.
The condition \eqref{MRS:eq:condconssi}
holds true if and only if
$\alpha > H - 1$.

\begin{corollary}
If $\alpha > H-1$,  the model \eqref{MRS:eq:ex-fbm} satisfies the conditions of
Theorem~\ref{MRS:thm:consdiscrL2}.
The estimator $\hat\theta^{(N)}$ in the model \eqref{MRS:eq:ex-fbm}
is $L_2$-consistent and strongly consistent
(provided that $\lim_{N\to\infty} t_N = +\infty$).
\end{corollary}

The condition $(B)$: $g\in L_2[0,T]$ holds true
if and only if $\alpha > -\frac12$.
The integral equation $\Gamma h = g$ is rewritten as
\[
\int_0^T \frac{h(s)\,ds}{|t-s|^{2H-2}} = \frac
{\alpha+1}{(2H-1)H} t^\alpha .
\]
If $\alpha > 2H-2$, then the solution is
\begin{equation}\label{MRS:eq:eqdefhpg}
h(t) = \textrm{const} \cdot
\left(
\frac{T^\alpha}{t^{H-\frac{1}{2}}  (T-t)^{H-\frac{1}{2}}} -
\alpha t^{\alpha+1-2H} W\left(\textstyle
\frac{T}{t},\; \alpha,\; H-\frac{1}{2}\right)
\right),
\end{equation}
where
$W\left(\textstyle
\frac{T}{t},\: \alpha,\: H-\frac{1}{2}\right)
= \int_0^{\frac{T}{t}-1}
(v+1)^{\alpha-1} v^{\frac{1}{2}-H} \, dv$.
The asymptotic behaviour of the function
$W\left(\frac{T}{t},\: \alpha, \: H-\frac12\right)$
as $t \to 0+$ is
\[
W\left(\textstyle
\frac{T}{t},\: \alpha,\: H-\frac{1}{2}\right)
\sim
\begin{cases}
\mathrm{B}\left(\frac32 - H, \: H - \frac12 - \alpha\right)
& \mbox{if $\alpha < H - \frac12$},\\
\ln(T/t)
& \mbox{if $\alpha = H - \frac12$},\\
\frac{2}{2\alpha+1-2H}
\frac{T^{\alpha-H+\frac12}}
     {t^{\alpha-H+\frac12}}
& \mbox{if $\alpha > H - \frac12$}.
\end{cases}
\]
Therefore, the function $h(t)$ defined in \eqref{MRS:eq:eqdefhpg}
is square integrable if
$
\alpha+1-2H - \max\left(0, \: \alpha - H + \frac12\right) > -\frac12
$,
which holds if $\alpha > 2H - \frac32$.
Note that if $\alpha > 2H - \frac32$, then the following inequalities
hold true:
$\alpha > 2H - 2$ (whence $h$ defined in \eqref{MRS:eq:eqdefhpg} is indeed a solution
to the integral equation $\Gamma h = g$),
$\alpha > H-1$ (whence conditions \eqref{MRS:eq:condconssi} and so \eqref{MRS:eq:ccondli0} are satisfied), and
$\alpha > -\frac12$ (whence condition (B) is satisfied).

\begin{corollary}
If $\alpha > 2 H - \frac32$, the conditions of
Theorems \ref{MRS:th-L}, \ref{MRS:thm:consTL2} and \ref{MRS:thm:thaxL2},
are satisfied.
The estimator $\hat\theta_T$ is $L_2$-consistent
and strongly consistent.
For fixed $T$, it can be approximated by discrete-sample estimator
in mean-square sense.
\end{corollary}

\subsection{The model with Brownian and fractional Brownian motion}
Consider the following model:
\begin{equation}\label{MRS:eq:mixed}
X_t = \theta t + W_t + B_t^{H},
\end{equation}
where $W$ is a standard Wiener process,
$B^{H}$ is a fractional Brownian motion
with Hurst index $H$,
and random processes $W$ and $B^{H}$ are independent.
The corresponding operator $\Gamma$ is
$\Gamma = I + \Gamma_H$,
where $\Gamma_H$ is defined by \eqref{MRS:eq:Gamma_H}.
The operator $\Gamma_H$ is self-adjoint and positive semi-definite.
Hence, the operator $\Gamma$ is invertible.  Thus
Assumption~(C) holds true.

In other words, the problem is reduced to the solving of the following Fredholm integral equation of the second kind
\begin{equation}\label{MRS:integralequationCCK}
h_T(u) +  H_2(2H_2-1)\int_0^T h_T(s)\absMRS{s-u}^{2H_2-2}\,ds=1,   \quad u\in [0,T].
\end{equation}
This approach to the drift parameter estimation in the model with mixed fractional Brownian motion was first developed in \cite{MRS:CaiChigKlept}.

Note also that the function $h_T = \Gamma_T^{-1} \constoneMRS$ can be evaluated iteratively
\begin{equation}\label{MRS:eq:approx}
h_T = \sum_{k=0}^\infty
\frac{\left( \frac12 \, \normMRS{\Gamma^H_T} \, I - \Gamma^H_T\right)^k \constoneMRS}
     {\left( 1 + \frac12 \, \normMRS{\Gamma^H_T}\right)^{k+1}} \,.
\end{equation}

\subsection{Model with subfractional Brownian motion}
\label{MRS:sec:examplesubfr}
\begin{definition}
The \emph{subfractional Brownian motion}\index{subfractional Brownian motion} $\widetilde B^H=\setMRS{\widetilde B^H_t,t\ge0}$  with Hurst parameter $H\in(0,1)$
is a centered Gaussian random process  with covariance function
\begin{equation}\label{MRS:eq:cov-sfbm}
\covMRS\left(\widetilde B^H_s,\widetilde B^H_t \right)= \frac
{2\,|t|^{2H} + 2\,|s|^{2H} - |t-s|^{2H} - |t+s|^{2H}}
2.
\end{equation}
\end{definition}
We refer \cite{MRS:BGT,MRS:Tudor07} for properties of this process.
Obviously, neither $\widetilde B^H\!$, nor its increments  are stationary.
If $\{B_t^H, \; t\in\mathbb{R}\}$ is
a fractional Brownian motion,
then the random process
$\frac{B^H_t + B^H_{-t}}{\sqrt{2}}$
is a subfractional Brownian motion.
Evidently, mixed derivative of the covariance function \eqref{MRS:eq:cov-sfbm} equals
\begin{equation}
\label{MRS:eq:K-subfr}
K_H(s,t) := \frac{\partial^2 \covMRS\left(\widetilde B^H_s, \widetilde B^H_t\right)}{\partial t \, \partial s} =
H \, (2H-1) \left(|t-s|^{2H-2} - |t+s|^{2H-2}\right) .
\end{equation}
If $H\in(\frac12,1)$,
then the operator $\Gamma = \widetilde\Gamma_H$ that satisfies
\eqref{MRS:eq:ass10st} for $\widetilde B^H$ equals
\begin{equation}\label{MRS:eq:oper-sfbm}
\widetilde\Gamma_H f(t) = \int_0^T K_H(s,t) f(s) \, ds .
\end{equation}

Consider the model \eqref{MRS:eq:model}
for $G(t) = t$ and $B=\widetilde B^H$:
\begin{equation}\label{MRS:model-linear-FBM}
X_t = \theta t + \widetilde B^H_t.
\end{equation}
Let us construct the estimators $\hat\theta^{(N)}$ and $\hat\theta_T$ from \eqref{MRS:eq:discr-mle} and \eqref{MRS:eq:defthetaC} respectively and establish their properties.
In particular, Proposition \ref{MRS:prop-sfbm-1} allows to define
finite-sample estimator $\hat\theta^{(N)}$.

\begin{proposition}
\label{MRS:prop-sfbm-1}
The linear equation $\widetilde\Gamma_H f = 0$
has only trivial solution in $L_2[0,T]$.
As a consequence, the finite slice
$\left(\widetilde B^H_{t_1}, \ldots, \widetilde B^H_{t_N}\right)$
with $0 < t_1 < \ldots < t_N$
has a multivariate normal distribution
with nonsingular covariance matrix.
\end{proposition}

Since $\varMRS\left(\widetilde B^H_t\right) = \left(2 - 2^{2 H -
1}\right) t^{2H}$, the random process $\widetilde B^H$ satisfies
Theorem~\ref{MRS:thm:consdiscrL2}. Hence, we have the following
result.
\begin{corollary}
Under condition $t_N\to +\infty$ as $N\to\infty$,
the estimator $\hat\theta^{(N)}$ in the model \eqref{MRS:model-linear-FBM} is $L_2$-consistent and
strongly consistent.
\end{corollary}

In order to define the continuous-time MLE~\eqref{MRS:eq:defthetaC},
we have to solve an integral equation.
The following statement guarantees
the existence of the solution.
\begin{proposition}\label{MRS:prop-sfbm-2}
If $\frac12 < H < \frac34$, then
the integral equation
$\widetilde\Gamma_H h = \indicatorMRS_{[0,T]}$, that is
\begin{equation}
\label{MRS:eq:ieqK}
\int_0^T K_H(s,t) h(s) \, ds = 1 \qquad
\mbox{for almost all $t \in (0,T)$}
\end{equation}
has a unique solution $h \in L_2[0,T]$.
\end{proposition}

\begin{corollary}
If $\frac12 < H < \frac34$, then
the random process $\widetilde B^H$ satisfies
Theorems \ref{MRS:th-L}, \ref{MRS:thm:consTL2}, and \ref{MRS:thm:thaxL2}.
As the result, $L(\theta)$ defined in \eqref{MRS:eq-L}
is the likelihood function
in the model \eqref{MRS:model-linear-FBM},
and $\hat\theta_T$ defined in \eqref{MRS:eq:defthetaC}
is the MLE.
The estimator is $L_2$-consistent and strongly consistent.
For fixed $T$, it can be approximated by discrete-sample estimator
in mean-square sense.
\end{corollary}

\subsection{The model with two independent fractional Brownian motions}
\label{MRS:sscec:2fBm}

Consider the following model:
\begin{equation}\label{MRS:eq:2fbm}
X_t = \theta t + B^{H_1}_t + B^{H_2}_t,
\end{equation}
where $B^{H_1}$ and $B^{H_2}$ are two independent fractional Brownian motion with Hurst indices
$H_1,H_2\in(\frac12,1)$.
Obviously, the condition~\eqref{MRS:eq:condconssi} is satisfied:
\[
\frac{\varMRS\left(B^{H_1}_t + B^{H_2}_t\right)}{t^2} =\frac{t^{2H_1}
+ t^{2H_2}}{t^2} \to 0, \quad t\to\infty.
\]
\begin{theorem}
Under condition $t_N\to +\infty$ as $N\to\infty$,
the estimator $\hat\theta^{(N)}$ in the model \eqref{MRS:eq:2fbm} is $L_2$-consistent and
strongly consistent.
\end{theorem}

Evidently, the corresponding operator $\Gamma$ for the model~\eqref{MRS:eq:2fbm} equals $\Gamma_{H_1}+\Gamma_{H_2}$, where $\Gamma_H$ is defined by \eqref{MRS:eq:Gamma_H}.
Therefore, in order to verify the assumptions of Theorem~\ref{MRS:th-L} we need to show that there exists a function $h_T$ such that
$\left(\Gamma_{H_1}+\Gamma_{H_2}\right) h_T = \indicatorMRS_{[0,T]}$.
This is equivalent to
$\left(I+\Gamma_{H_1}^{-1}\Gamma_{H_2}^{}\right) h_T = \Gamma_{H_1}^{-1}\indicatorMRS_{[0,T]}$,
since the operator $\Gamma_{H_1}$ is injective and its range contains $\indicatorMRS_{[0,T]}$,
see \eqref{MRS:eq:GHhT1} and Theorem~\ref{MRS:lemma:soltoiepk}.
Hence, it suffices to prove that the operator
$I+\Gamma_{H_1}^{-1}\Gamma_{H_2}$ is
invertible. This is done in Theorem~\ref{MRS:th:boundedness} in
Appendix~\ref{MRS:th:boundedness} for $H_1\in(1/2,3/4]$,
$H_2\in(H_1,1)$. Thus, in this case the assumptions of
Theorem~\ref{MRS:th-L} hold with
\[
h_T =  \left(I+\Gamma_{H_1}^{-1}\Gamma_{H_2}\right)^{-1}\Gamma_{H_1}^{-1}\indicatorMRS_{[0,T]}.
\]
Therefore, we have the following result for the estimator
\[
\hat\theta_T = \frac
{\int_0^T h_T(s) \, dX_s}
{\int_0^T h_T(s) \, ds} .
\]
\begin{theorem}
If $H_1\in(1/2,3/4]$ and $H_2\in(H_1,1)$, then
the random process $B^{H_1} + B^{H_2}$ satisfies
Theorems \ref{MRS:th-L}, \ref{MRS:thm:consTL2}, and \ref{MRS:thm:thaxL2}.
As the result, $L(\theta)$ defined in \eqref{MRS:eq-L}
is the likelihood function
in the model \eqref{MRS:eq:2fbm},
and $\hat\theta_T$
is the maximum likelihood estimator.
The estimator is $L_2$-consistent and strongly consistent.
For fixed $T$, it can be approximated by discrete-sample estimator
in mean-square sense.
\end{theorem}

\begin{remark}\label{MRS:rem:mish-vor}
Another approach to the drift parameter estimation in the model with two fractional Brownian motions was proposed in \cite{MRS:MiSumfbms1} and developed in \cite{MRS:MishuraVoronov}.
It is based on the solving of the following Fredholm integral equation of the second kind
\begin{equation}\label{MRS:integralequation}
(2-2H_1) \tilde h_T(u)u^{1-2H_1}
+\int_0^T \tilde h_T(s) k(s,u)\,ds
= (2-2H_1)  u^{1-2H_1}, \quad u\in(0,T],
\end{equation}
where
\begin{gather*}
    k(s,u) =\int_0^{s\wedge u}\partial_s K_{H_1,H_2}(s,v) \partial_u K_{H_1,H_2}(u,v)\,dv,
    \\
    K_{H_1,H_2}(t,s) = c_{H_1} \beta_{H_2} s^{1/2-H_2}\int_{s}^{t}(t-u)^{1/2-H_1}u^{H_2-H_1}(u-s)^{H_2-3/2}du,
\\
c_{H_1}=\left(\frac{\Gamma(3-2H_1)}{2H_1\Gamma(\frac32-H_1)^{3}\Gamma(H_1+\frac12)}\right)^{\frac{1}{2}}\!,
\; \beta_{H_2}= \left(\frac{2H_2\left(H_2-\frac12\right)^2\Gamma(\frac32-H_2)}{\Gamma(H_2+\frac12)\Gamma(2-2H_2)}\right)^{\frac12}.
    \end{gather*}
Then for $1/2\le H_1<H_2<1$ the estimator is defined as
\begin{equation*}
\hat{\theta}(T)=\frac{N(T)}{\delta_{H_1}\langle N\rangle(T)},
\end{equation*}
where
$\delta_{H_1}=c_{H_1}\BetaMRS\left(\frac32-H_1,\frac32-H_1\right)$,
$N(t)$
is a square integrable Gaussian martingale,
\[
N(T)=\int_0^{T}\tilde h_{T}(t)\,dX(t),
\]
$\tilde h_{T}(t)$ is a unique solution to \eqref{MRS:integralequation} and \begin{equation*}
\langle N\rangle(T)=(2-2H_1) \int_0^{T} \tilde  h_{T}(t)t^{1-2H_1}\,dt.
\end{equation*}
This estimator is also unbiased, normal and strongly consistent.
The details of this method can be found also in \cite[Sec.~5.5]{MRS:KMR-book}.
\end{remark}

\appendix

\section{Integral equation with power kernel}
\label{MRS:sec:appie}
\begin{theorem}\label{MRS:lemma:soltoiepk}
Let $0< p < 1$ and $b>0$.
\begin{enumerate}
\item
If $y \in L_1[0,b]$ is a solution to integral equation
\begin{equation}
\label{MRS:eq:ieq1}
\int_0^b \frac{y(s)\, ds}{|t-s|^p} = f(t)
\quad \mbox{for almost all $t\in(0,b)$,}
\end{equation}
then $y(x)$ satisfies
\begin{equation}\label{MRS:eq:solieq1}
y(x) = \frac{\Gamma(p) \cos\frac{\pi p}{2}}
{\pi x^{(1-p)/2}}
\mathcal{D}^{(1-p)/2}_{b-}
\left(
x^{1-p}
\mathcal{D}^{(1-p)/2}_{0+}
\left( \frac{f(x)}{x^{(1-p)/2}} \right)
\right)
\end{equation}
almost everywhere on $[0,b]$,
where $\mathcal{D}^{\alpha}_{a+}$ and $\mathcal{D}^{\alpha}_{b-}$
are the Riemann--Liouville fractional derivatives, that is
\begin{align*}
\mathcal{D}^{\alpha}_{a+} f(x) &=
\frac{1}{\Gamma(1-\alpha)} \frac{d}{dx} \biggl(\int_a^x \frac{f(t)}{(x-t)^\alpha} \, dt\biggr),
\\
\mathcal{D}^{\alpha}_{b-} f(x) &=
\frac{-1}{\Gamma(1-\alpha)} \frac{d}{dx} \biggl(\int^b_x \frac{f(t)}{(t-x)^\alpha} \, dt\biggr).
\end{align*}

\item
If $y_1 \in L_1[0,b]$ and $y_2 \in L_1[0,b]$ are two solutions
to integral equation \eqref{MRS:eq:ieq1},
then $y_1(x) = y_2(x)$
almost everywhere on $[0,b]$.
\item
If $y \in L_1[0,b]$ satisfies \eqref{MRS:eq:solieq1}
almost everywhere on $[0,b]$
and the fractional derivatives are solutions to respective Abel integral equations,
that is
\begin{equation}\label{MRS:eq:ID1}
\frac{1}{\Gamma\left(\frac{1-p}{2}\right)}
\int_0^t
\frac{\mathcal{D}^{(1-p)/2}_{0+}
(f(x)  x^{(p-1)/2})}
{(t-x)^{(p+1)/2}} \, dx =
\frac{f(t)}{t^{(1-p)/2}},
\end{equation}
for almost all $t \in (0,b)$ and
\begin{equation}\label{MRS:eq:ID2}
\frac{1}{\Gamma\left(\frac{1-p}{2}\right)}
\int_x^b
\frac{\pi y(s) s^{(1-p)/2}}{\Gamma(p) \cos\frac{\pi p}{2}}\,
\frac{ds}{(s-x)^{(p+1)/2}}
=
x^{1-p}
\mathcal{D}^{(1-p)/2}_{0+}
\left( \frac{f(x)}{x^{(1-p)/2}} \right)
\end{equation}
for almost all $x \in (0,b)$,
then $y(s)$ is a solution to integral equation \eqref{MRS:eq:ieq1}.
\end{enumerate}
\end{theorem}

\begin{proof}
Firstly, transform the left-hand side of
\eqref{MRS:eq:ieq1}.
By \cite[Lemma 2.2(i)]{MRS:nvv}, for $0 < s < t$
\[
\int_0^s \frac{d\tau}{
(t-\tau)^{(p+1)/2}
(s-\tau)^{(p+1)/2}
\tau^{1-p} }
= \frac{\BetaMRS\left(p, \: \frac{1-p}{2}\right)}
{s^{(1-p)/2}
t^{(1-p)/2}
(t-s)^p} .
\]
Hence, for $s>0$, $t>0$, $s\neq t$
\[
\int_0^{\min (s,t)} \frac{d\tau}{
(t-\tau)^{(p+1)/2}
(s-\tau)^{(p+1)/2}
\tau^{1-p} }
= \frac{\BetaMRS\left(p, \: \frac{1-p}{2}\right)}
{s^{(1-p)/2}
t^{(1-p)/2}
|t-s|^p} .
\]
Hence
\begin{multline*}
\int_0^b \frac {y(s)\, ds}{|t-s|^p}
= \\ =
\int_0^b \frac{
s^{(1-p)/2}
t^{(1-p)/2}
y(s)}
{\BetaMRS\left(p, \: \frac{1-p}{2}\right)}
\int_0^{\min(s,t)}
\frac{d\tau}{
(t-\tau)^{(p+1)/2}
(s-\tau)^{(p+1)/2}
\tau^{1-p} }
\, ds .
\end{multline*}
Change the order of integration, noting that
$\{(s, \tau) :\allowbreak 0 < s < b,\allowbreak \; 0 < \tau < \min(s, t)\} =
\allowbreak
\{(s, \tau) :\allowbreak 0 < \tau < t,\allowbreak \; \tau < s < b\}$
for $0<t<b$:
\begin{equation}\label{MRS:eq:int2D}
\int_0^b \frac {y(s)\, ds}{|t-s|^p}
=
\frac{t^{(1-p)/2}}
{\BetaMRS\left(p, \: \frac{1-p}{2}\right)}
\int_0^t \frac{1}
{(t-\tau)^{(p+1)/2} \tau^{1-p}}
\int_{\tau}^b
\frac{
s^{(1-p)/2}
y(s) \, ds}
{(s-\tau)^{(p+1)/2} }
\, d\tau.
\end{equation}
The right-hand side of \eqref{MRS:eq:int2D} can be rewritten with
fractional integration:
\[
\int_0^b \frac {y(s)\, ds}{|x-s|^p}
=
\frac{\Gamma\left(\frac{1-p}{2}\right)^2}
{\BetaMRS\left(p, \: \frac{1-p}{2}\right)}
x^{(1-p)/2}
I^{(1-p)/2}_{0+} \left(
\frac{1}{x^{1-p}}
I^{(1-p)/2}_{b-}
( x^{(1-p)/2} y(x) )
\right)
\]
for $0 < x < b$, where $I^\alpha_{a+}$ and $I^\alpha_{b-}$
are fractional integrals
\[
I^{\alpha}_{a+} f(x) =
\frac{1}{\Gamma(\alpha)}  \int_a^x \frac{f(t)}{(x-t)^{(1-\alpha)}} \, dt
\]
and
\[
I^{\alpha}_{b-} f(x) =
\frac{1}{\Gamma(\alpha)} \int^b_x \frac{f(t)}{(t-x)^{(1-\alpha)}} \, dt.
\]
The constant coefficient can be simplified:
\[
\frac{\Gamma\left(\frac{1-p}{2}\right)^2}
{\BetaMRS\left(p, \: \frac{1-p}{2}\right)}
=
\frac{\Gamma\left(\frac{1-p}{2}\right) \Gamma\left(\frac{p+1}{2}\right)}
{\Gamma(p)}
=
\frac{\pi}
{\Gamma(p) \cos\left(\frac{\pi p}{2}\right)} .
\]

Thus integral equation \eqref{MRS:eq:ieq1} can be rewritten with use of fractional integrals:
\begin{equation}\label{MRS:eq:ieq1fi}
\frac{\pi}
{\Gamma(p) \cos\left(\frac{\pi p}{2}\right)}
x^{(1-p)/2}
I^{(1-p)/2}_{0+} \left(
\frac{1}{x^{1-p}}
I^{(1-p)/2}_{b-}
( x^{(1-p)/2} y(x) )
\right)
= f(x)
\end{equation}
for almost all $x \in (0,b)$.

Whenever $y \in L_1[0,b]$, the function $x^{(1-p)/2} y(x)$
is obviously integrable on $[0,b]$.
Now prove that the function
$x^{p-1} I^{(1-p)/2}_{b-}
( x^{(1-p)/2} y(x) )$
is also integrable on $[0,b]$.
Indeed,
\begin{gather*}
\left|I^{(1-p)/2}_{b-}
( x^{(1-p)/2} y(x) )\right| \le
\frac{1}{\Gamma\left(\frac{1-p}{2}\right)}
\int^b_x \frac{t^{(1-p)/2} \, |y(t)|}{(t-x)^{(p+1)/2}} \, dt ,
\\
\begin{aligned}
&\int_0^b \left|
\frac{I^{(1-p)/2}_{b-}
( x^{(1-p)/2} y(x) )}
{x^{1-p}}
\right| dx
\le
\frac{1}{\Gamma\left(\frac{1-p}{2}\right)}
\int_0^b \frac{1}{x^{1-p}}
\int^b_x \frac{t^{(1-p)/2} \, |y(t)|\, dt}{(t-x)^{(p+1)/2}}
\,dx
\\
&\qquad=
\frac{1}{\Gamma\left(\frac{1-p}{2}\right)}
\int^b_0 t^{(1-p)/2} \, |y(t)|\,
\int_0^t \frac{dx}{x^{1-p} (t-x)^{(p+1)/2}}
\, dt
\\
&\qquad=
\frac{\BetaMRS\left(p,\: \frac{1-p}{2}\right)}
{\Gamma\left(\frac{1-p}{2}\right)}
\int^b_0  |y(t)|\, dt
< \infty ,
\end{aligned}
\end{gather*}
and the integrability is proved.

Due to \cite[Theorem 2.1]{MRS:samko1993fractional},
the Abel integral equation
$f(x) = I^\alpha_{a+} \phi(x)$, $x \in (a,b)$,
may have not more that one solution
$\phi(x)$ within $L_1[a,b]$.
If the equation has such a solution, then the solution
$\phi(x)$ is equal to $\mathcal{D}^\alpha_{a+} f(x)$.
Similarly,
the Abel integral equation
$f(x) = I^\alpha_{b-} \phi(x)$
may have not more that one solution
$\phi(x) \in L_1[a,b]$,
and if it exists, $\phi = \mathcal{D}^\alpha_{b-} f$.

Therefore, if $y\in L_1[0,b]$ is a solution to integral equation,
then is also satisfies \eqref{MRS:eq:ieq1fi}, so
\begin{gather}
I^{(1-p)/2}_{0+} \left(
\frac{1}{x^{1-p}}
I^{(1-p)/2}_{b-}
\left(
\frac{\pi}
{\Gamma(p) \cos\left(\frac{\pi p}{2}\right)}
 x^{(1-p)/2} y(x) \right)
\right)
= \frac{f(x)}{x^{(1-p)/2}},
\label{MRS:eq:rewreq2}\\
\frac{1}{x^{1-p}}
I^{(1-p)/2}_{b-}
\left(
\frac{\pi  x^{(1-p)/2} y(x)}
{\Gamma(p) \cos\left(\frac{\pi p}{2}\right)}
 \right)
=
\mathcal{D}^{(1-p)/2}_{0+}
\left(\frac{f(x)}{x^{(1-p)/2}}\right),
\nonumber\\
\frac{\pi  x^{(1-p)/2} y(x)}
{\Gamma(p) \cos\left(\frac{\pi p}{2}\right)}
=
\mathcal{D}^{(1-p)/2}_{b-} \left(
x^{1-p}
\mathcal{D}^{(1-p)/2}_{0+}
\left(\frac{f(x)}{x^{(1-p)/2}}\right)
\right)
\nonumber
\end{gather}
 for almost all $x \in (0, b)$.
Thus $y(x)$ satisfies \eqref{MRS:eq:solieq1}.
Statement 1 of Theorem~\ref{MRS:lemma:soltoiepk} is proved,
and statement 2 follows from statement 1.

From equations \eqref{MRS:eq:ID1} and \eqref{MRS:eq:ID2},
which can be rewritten with fractional integration operator,
\begin{gather*}
I^{(1-p)/2}_{0+}
\mathcal{D}^{(1-p)/2}_{0+}
(f(x)  x^{(1-p)/2}) =
\frac{f(x)}{t^{(1-p)/2}},
\\
I^{(1-p)/2}_{b-}
\left(
\frac{\pi y(x) x^{(1-p)/2}}{\Gamma(p) \cos\frac{\pi p}{2}}
\right)
=
x^{1-p}
\mathcal{D}^{(1-p)/2}_{0+}
\left( \frac{f(x)}{x^{(1-p)/2}} \right)
\end{gather*}
\eqref{MRS:eq:rewreq2} follows, and
\eqref{MRS:eq:rewreq2} is equivalent to \eqref{MRS:eq:ieq1}.
Thus statement 3 of Theorem~\ref{MRS:lemma:soltoiepk} holds true.
\qed
\end{proof}

\begin{remark}
The integral equation \eqref{MRS:eq:ieq1} was solved explicitly in \cite[Lemma 3]{MRS:Breton98} under the assumption $f\in C([0,b])$.
Here we solve this equation in $L_1[0,b]$ and prove the uniqueness of a solution in this space.
Note also that the formula for solution in the handbook \cite[formula 3.1.30]{MRS:Polyanin_Manzhirov} is incorrect
(it is derived from the incorrect formula 3.1.32 of the same book, where an operator of differentiation is missing; this error comes from the book \cite{MRS:Zabreyko}).
\end{remark}

\section{Boundedness and invertibility of operators}\label{MRS:Apx:Boundness}

This appendix is devoted to the proof of the following result, which plays the key role in the proof of the strong consistency of the MLE for the model with two independent fractional Brownian motions.
\begin{theorem}\label{MRS:th:boundedness}
    Let $H_1\in\left(\frac12,\frac34\right]$, $H_2\in(H_1,1)$, and $\Gamma_H$ be the operator defined by \eqref{MRS:eq:Gamma_H}.
    Then  $\Gamma^{-1}_{H_1}\Gamma_{H_2}\colon L_2[0,T]\to L_2[0,T]$ is a compact linear operator
defined on the entire space $L_2[0,T]$,
and the operator $I+\Gamma_{H_1}^{-1}\Gamma_{H_2}$ is invertible.
\end{theorem}
The proof consists of several steps.

\subsection{Convolution operator}
If $\phi \in L_1[-T,T]$, then the following convolution operator
\begin{equation}\label{MRS:eq:convopL}
Lf(x) = \int_0^T \phi(t-s) f(s) \, ds
\end{equation}
is a linear continuous operator
$L_2[0,T] \to L_2[0, T]$,
and
\begin{equation}\label{MRS:eq:boconvnorm}
\|L\| \le  \int_{-T}^T |\phi(t)| \, dt.
\end{equation}
Moreover, $L$  is a compact operator.

The adjoint operator of the operator \eqref{MRS:eq:convopL} is
\[
L^*f(x) = \int_0^T \phi(s-t) f(s) \, ds .
\]
If the function $\phi$ is even, then the linear operator $L$ is self-adjoint.

Let us consider the following convolution operators.
\begin{definition}
For $\alpha>0$, the Riemann--Liouville operators of fractional integration  are defined as
\begin{align*}
I^\alpha_{0+} f(t) = \frac{1}{\Gamma(\alpha)} \int_0^t \frac{f(s)\, ds}{(t-s)^{1-\alpha}}, \\
I^\alpha_{T-} f(t) = \frac{1}{\Gamma(\alpha)} \int^T_t \frac{f(s)\, ds}{(t-s)^{1-\alpha}}.
\end{align*}
The operators  $I^\alpha_{0+}$ and $I^\alpha_{T-}$ are mutually adjoint.
Their norm can be bounded as follows
\begin{equation}\label{MRS:neq:frInormbo}
\|I^\alpha_{T-}\| = \|I^\alpha_{0+}\| \le
\frac{1}{\Gamma(\alpha)} \int_0^T \frac{ds}{s^{1-\alpha}}
=\frac{T^\alpha}{\Gamma(\alpha+1)} .
\end{equation}

Let $\frac12 < H < 1$ and $\Gamma_H$ be the operator defined by \eqref{MRS:eq:Gamma_H}.
Then
\begin{equation}\label{MRS:eq:GammaThruIpI}
\Gamma_H 
= H\Gamma(2H) \left(I^{2H-1}_{0+} + I^{2H-1}_{T-}\right).
\end{equation}

The linear operators $I^\alpha_{0+}$,  $I^\alpha_{T-}$ for $\alpha > 0$, and $\Gamma_H$ for $\frac12 < H < 1$
are injective.
\end{definition}

\subsection{Semigroup property of the operator of fractional integration}
\begin{theorem}\label{MRS:thm:IeSemi}
For $\alpha>0$ and $\beta>0$ the following equalities hold
\begin{gather*}
I^\alpha_{0+} I^\beta_{0+} = I^{\alpha+\beta}_{0+}, \\
I^\alpha_{T-} I^\beta_{T-} = I^{\alpha+\beta}_{T-}.
\end{gather*}
\end{theorem}
This theorem is a particular case of \cite[Theorem~2.5]{MRS:samko1993fractional}.

\begin{proposition}\label{MRS:prop:acutea}
For $0 < \alpha \le \frac12$ and $f\in L_2[0,T]$,
\[
\langle I^\alpha_{0+} f, \: I^\alpha_{T-}f \rangle \ge 0.
\]
Equality is achieved if and only if
\begin{itemize}
\item
$f = 0$ almost everywhere on $[0, T]$ for $0 < \alpha < \frac12$;
\item
$\int_0^T f(t)\, dt = 0$  for $\alpha = \frac12$.
\end{itemize}
\end{proposition}
 \begin{proof}
Since the operators $I^\alpha_{0+}$ and $I^\alpha_{T-}$ are mutually adjoint, by semigroup property, we have that
\begin{gather*}
\langle I^\alpha_{0+} f, \: I^\alpha_{T-}f \rangle =
\langle I^\alpha_{0+} I^\alpha_{0+} f, f \rangle =
\langle I^{2\alpha}_{0+} f, f \rangle, \\
\langle I^\alpha_{0+} f, \: I^\alpha_{T-}f \rangle =
\langle f, I^\alpha_{T-} I^\alpha_{T-}f \rangle =
\langle f, I^{2\alpha}_{T-}f \rangle.
\end{gather*}
Adding these equalities, we obtain
\begin{equation}\label{MRS:eq:halfint}
\langle I^\alpha_{0+} f, \: I^\alpha_{T-}f \rangle =
\frac12
\langle I^{2\alpha}_{0+} f + I^{2\alpha}_{T-}f, \: f \rangle.
\end{equation}

If $0 < \alpha < \frac12$, then
\begin{align*}
\langle I^\alpha_{0+} f, \: I^\alpha_{T-}f \rangle &=
\frac{1}{2 H \Gamma(2 H)}\,
\langle \Gamma_H f, \: f \rangle
= \\ &=
\frac{1}{2 H \Gamma(2 H)}
\MEMRS \biggl(\int_0^T f(t) \, dB_t^H \biggr)^2 \ge 0.
\end{align*}
where $H = \alpha + \frac12$, $\frac12 < H < 1$, and
$B_t^H$ is a fractional Brownian motion.

Let us consider the case $\alpha = \frac12$.
Since
\begin{align*}
I^1_{0+} f (t) + I^1_{T-}f (t)
&= \int_0^t f(s)\, ds + \int_t^T f(s)\, ds =
\int_0^T f(s)\, ds,
\\
I^1_{0+} f + I^1_{T-}f
&= \int_0^T f(s)\, ds \, \indicatorfMRS_{[0,T]},
\\
\left \langle I^1_{0+} f + I^1_{T-}f, \: f \right \rangle &=
\int_0^T f(s)\, ds \: \langle   \indicatorfMRS_{[0,T]}, \: f\rangle =
\biggl( \int_0^T f(s)\, ds \biggr) ^2,
\end{align*}
we see from \eqref{MRS:eq:halfint} that
\begin{equation}\label{MRS:neq:line1163}
\left \langle I^{1/2}_{0+} f, \: I^{1/2}_{T-}f \right \rangle =
\frac12 \biggl( \int_0^T f(s)\, ds \biggr)^2  \ge 0.
\end{equation}

Conditions for the equality
$\langle I^\alpha_{0+} f, \: I^\alpha_{T-}f \rangle = 0$
can be easily found by analyzing the proof.
  Indeed, if $0 < \alpha < \frac12$ and $H = \alpha + \frac12$,
  then $\Gamma_H$ is a self-adjoint positive compact operator
  whose eigenvalues are all positive.
  Then $2 H \Gamma(2H) \langle I^\alpha_{0+} f, \: I^\alpha_{T-}f \rangle
  = \langle \Gamma_H f, \: f \rangle
  = \| \Gamma_H^{1/2} f \|^2$.
  In this case, the equality $\langle I^\alpha_{0+} f, \: I^\alpha_{T-}f \rangle = 0$
  holds true if and only if $f = 0$ almost everywhere on $[0,T]$.
  If $\alpha = \frac12$, then the condition for the equality
  follows from \eqref{MRS:neq:line1163}. \qed
\end{proof}

\begin{proposition}\label{MRS:prop:sunonosu}
For $0 < \alpha \le \frac12$ and $f \in L_2[0,T]$,
\begin{equation}\label{MRS:neq:sunonosu1}
\| I^\alpha_{0+} f \| + \| I^\alpha_{T-}f \| \le \sqrt{2}\,
\| I^\alpha_{0+} f + I^\alpha_{T-}f \|.
\end{equation}
Consequently, for $\frac 12 < H \le \frac34$
\[
\left \| I^{2H-1}_{0+} f \right \| + \left \| I^{2H-1}_{T-}f \right \| \le
\frac{\sqrt{2}}{H \Gamma(2H)}
\| \Gamma_H f \| .
\]
\end{proposition}
\begin{proof}
Taking into account Proposition~\ref{MRS:prop:acutea}, we get
\begin{align*}
(\| I^\alpha_{0+} f \| + \| I^\alpha_{T-}f \|)^2
&\le
2\,\| I^\alpha_{0+} f \|^2 + 2\,\| I^\alpha_{T-}f \|^2
\\ &\le
2\,\| I^\alpha_{0+} f \|^2 + 4 \, \langle I^\alpha_{0+} f, \: I^\alpha_{T-}f \rangle  + 2\,\| I^\alpha_{T-}f \|^2
\\ &=
2 \, \| I^\alpha_{0+} f + I^\alpha_{T-}f \|^2 ,
\end{align*}
whence the inequality \eqref{MRS:neq:sunonosu1} follows.

The second statement is obtained by the representation~\eqref{MRS:eq:GammaThruIpI}.\qed
\end{proof}

\subsection{Transposition of operators}

\begin{lemma}\label{MRS:lem:optra}
Let $A$ be a linear continuous operator on $L_2[0,T]$, and
$B$ be an injective self-adjoint compact linear operator on $L_2[0,T]$.
If the linear operator $A^* B^{-1}$ is bounded, that is
$\|A^* B^{-1}\| = K < \infty$, then
the linear operator $B^{-1} A$ is defined on the entire space
$L_2[0,T]$, bounded, and $\|B^{-1} A\| = K$.
\end{lemma}
\begin{proof}
For the self-adjoint compact linear operator $B$
one can find an orthonormal  eigenbasis
$\{e_1, e_2, \ldots\}$ such that
\[
B\biggl( \sum_{k=1}^\infty x_k e_k \biggr) =
\sum_{k=1}^\infty \lambda_k x_k e_k
\quad
\text{for}
\quad
\sum_{k=1}^\infty x_k^2 < +\infty.
\]
Then $\lim_{k \to \infty} \lambda_k = 0$
(by compactness), but
for all $k$ the inequality $\lambda_k \neq 0$ holds
(by injectivity).

The inverse operator is a self-adjoint linear operator defined
by the equation
\[
B^{-1} \biggl( \sum_{k=1}^\infty x_k e_k \biggr) =
\sum_{k=1}^\infty \frac{x_k}{\lambda_k} e_k
\quad
\text{for}
\quad
\sum_{k=1}^\infty \frac{x_k^2}{\lambda_k^2} < +\infty.
\]
The domain of the operator $B^{-1}$ is the subset
\[
B(L_2[0,T]) = \left\{
\sum_{k=1}^\infty x_k e_k : \sum_{k=1}^\infty \frac{x_k^2}{\lambda_k^2} < \infty\right\}
\]
of the Hilbert space $L_2[0,T]$.

Let us prove that the operator $B^{-1} A$ is defined on $L_2[0,T]$.  Assume the opposite, i.\,e.,
$B^{-1} A$ is undefined at some point $f \in L_2[0, T]$.
This means that $A f \not\in B(L_2[0,T])$.

Decompose $A f$ into a series by the eigenfunctions of the operator $B$:
\begin{equation}\label{MRS:eq:decomp}
A f = \sum_{k=1}^\infty x_k e_k .
\end{equation}
Since $A f \not\in B(L_2[0,T])$, we see that
\[
\sum_{k=1}^\infty \frac{x_k^2}{\lambda_k^2} = +\infty,
\]
and for
\[
s_n = \sum_{k=1}^n \frac{x_k^2}{\lambda_k^2}
\]
it holds that $\lim_{n\to\infty} s_n = +\infty$
and $s_n \ge 0$ for all $n\in\mathbb{N}$.
Therefore, there exists $N\in\mathbb{N}$ such that
\begin{equation}\label{MRS:neq:su249a1}
s_N > K^2 \, \|f\|^2  .
\end{equation}

Put
\[
g = \sum_{k=1}^N \frac{x_k}{\lambda_k} e_k .
\]
Then
\begin{gather*}
\|g\|^2 = \sum_{k=1}^N \frac{x_k^2}{\lambda_k^2} = s_N, \qquad
\|g\| = \sqrt{s_k}, \qquad
g \in B(L_2[0, T]), \\
B^{-1} g = \sum_{k=1}^N \frac{x_k}{\lambda_k^2} e_k, \qquad
\left \langle A^* B^{-1} g, \: f \right \rangle =
\left \langle B^{-1} g, A f \right \rangle
= \sum_{k=1}^N  \frac{x_k}{\lambda_k^2} \, x_k = s_N .
\end{gather*}
By the Cauchy--Schwarz inequality,
\[
\left| \left\langle A^* B^{-1} g, \: f \right\rangle \right| \le
\left \| A^* B^{-1} g \right \| \, \| f \| \le \| A^* B^{-1} \| \,
\|g\| \, \| f \| = K \, \sqrt{s_N} \, \|f\| .
\]
Hence,
\begin{equation}\label{MRS:neq:su249a2}
s_N \le K \, \sqrt{s_N} \, \|f\| .
\end{equation}
The inequalities \eqref{MRS:neq:su249a1} and \eqref{MRS:neq:su249a2}
contradict each other.
Thus, the operator $B^{-1} A$ is defined on the entire space $L_2[0,T]$.

Now let us prove boundedness of the operator $B^{-1} A$ and the
inequality \linebreak $\left\| B^{-1} A \right\| \le K$. Suppose
that this is not so. Then there exists an element $f$ of the space
$L_2[0,T]$ such that
\begin{equation}\label{MRS:neq:su249a30}
\left \|B^{-1} A f\right \| > K \|f\|.
\end{equation}

We use the same decomposition of the vector $Af$ into the eigenvectors of $B$ as above, see \eqref{MRS:eq:decomp}. Then
\begin{gather*}
 B^{-1} A f = \sum_{k=1}^\infty \frac{x_k}{\lambda_k} e_k, \qquad
 \left \| B^{-1} A f \right \|^2 = \sum_{k=1}^\infty \frac{x_k^2}{\lambda_k^2}, \\
 \lim_{n\to\infty} s_n = \left \| B^{-1} A f \right \|^2 > K^2 \|f\|^2,
\end{gather*}
by~\eqref{MRS:neq:su249a30}.
Therefore, there exists $N\in\mathbb{N}$ such that the inequality~\eqref{MRS:neq:su249a1} holds.
Arguing as above, we get a contradiction.
Hence, $\|B^{-1} A\| \le K$.

It remains to prove the opposite inequality $\|B^{-1} A\| \ge K$.
The operator $A^* B^{-1}$ is defined on the set $B(L_2[0,T])$.
For all $f \in B(L_2[0,T])$ from the domain of the operator
$A^* B^{-1}$, we have
\begin{align*}
\| A^* B^{-1} f \|^2
&= \langle B^{-1} A A^* B^{-1} f, \: f \rangle
\le \\  &\le
\| B^{-1} A A^* B^{-1} f \| \, \|f\|
\le
\| B^{-1} A \| \, \|A^* B^{-1} f \| \, \|f\|,
\end{align*}
whence
\[
\| A^* B^{-1} f \| \le \| B^{-1} A \|  \, \|f\| .
\]
Therefore
$K = \| A^* B^{-1} \| \le \| B^{-1} A \|$.\qed
\end{proof}

\subsection{The proof of boundedness and compactness}
\begin{proposition}
Let $\frac12 < H_1 < H_2 < 1$ and $H_1 \le \frac34$. Then
$\Gamma_{H_1}^{-1} \Gamma_{H_2}^{}$ is a compact linear operator
defined on the entire space $L_2[0,T]$.
\end{proposition}

\begin{proof}
The operator $\Gamma_{H_1}$ is an injective self-adjoint compact operator $L_2[0,T] \to L_2[0,T]$.
The inverse operator $\Gamma_{H_1}^{-1}$ is densely defined on $L_2[0,T]$.
By Proposition~\ref{MRS:prop:sunonosu}, the operators $I^{2H_1-1}_{0+} \Gamma_{H_1}^{-1}$
and $I^{2H_1-1}_{T-} \Gamma_{H_1}^{-1}$ are bounded.
Therefore, by Lemma~\ref{MRS:lem:optra}, the operators $\Gamma_{H_1}^{-1} I^{2H_1-1}_{T-}$
and $\Gamma_{H_1}^{-1} I^{2H_1-1}_{0+}$ are also bounded and defined on the entire space $L_2[0,T]$.
By \eqref{MRS:eq:GammaThruIpI} and the semigroup property (Theorem~\ref{MRS:thm:IeSemi}),
\begin{align*}
\Gamma_{H_1}^{-1} \Gamma_{H_2}^{}
&=
H\Gamma(2H) \left( \Gamma_{H_1}^{-1} I^{2H_2-1}_{0+} + \Gamma_{H_1}^{-1} I^{2H_2-1}_{T-} \right)
= \\ &=
H\Gamma(2H) \left( \Gamma_{H_1}^{-1} I^{2H_1-1}_{0+} I^{2(H_2-H_1)}_{0+} + \Gamma_{H_1}^{-1} I^{2H_1-1}_{T-} I^{2(H_2-H_1)}_{T-}\right) .
\end{align*}
Since $I^{2(H_2-H_1)}_{0+}$ and $I^{2(H_2-H_1)}_{T-}$ are compact operators,
the operator $\Gamma_{H_1}^{-1} \Gamma_{H_2}^{}$ is also compact.\qed
\end{proof}

\subsection{The proof of invertibility}
Now prove that $-1$ is not an eigenvalue of the linear operator $\Gamma_{H_1}^{-1}\Gamma_{H_2}^{}$.
Indeed, if $\Gamma_{H_1}^{-1}\Gamma_{H_2}^{} f = -f$ for some function $f \in L_2[0,T]$,
then $\Gamma_{H_2} f + \Gamma_{H_1} f = 0$.
Since $\Gamma_{H_2}$ and $\Gamma_{H_1}$ are positive definite self-adjoint (and injective) operators,
$\Gamma_{H_2} + \Gamma_{H_1}$ is also a positive definite self-adjoint and injective operator.
Hence $f = 0$ almost everywhere on $[0, T]$.

Because $-1$ is not an eigenvalue of the compact linear operator $\Gamma_{H_1}^{-1}\Gamma_{H_2}^{}$,
$-1$ is a regular point, i.e., $-1 \not\in \sigma(\Gamma_{H_1}^{-1}\Gamma_{H_2}^{})$,
and the linear operator $\Gamma_{H_1}^{-1}\Gamma_{H_2}^{} + I$ is invertible.

\bigskip
\begin{petit}
\noindent
\textbf{Acknowledgements }
The research of Yu.~Mishura was funded (partially) by the Australian Government through the Australian Research Council (project number DP150102758).
Yu.~Mishura and K. Ralchenko acknowledge that the present research is carried through within the frame and support of the ToppForsk project nr. 274410 of the Research Council of Norway with title STORM: Stochastics for Time-Space Risk Models.
\end{petit}

\bibliographystyle{spmpsci}
\bibliography{biblio}

\end{document}